\documentclass[a4paper]{amsart}
\usepackage{amssymb}
\usepackage{amsthm}
\usepackage{amsmath,amscd}
\usepackage[mathscr]{euscript} 
\usepackage{mathrsfs}
\usepackage[all]{xy}
\usepackage{color}
\usepackage[utf8]{inputenc}
\usepackage{mathptmx} 
\usepackage[scaled=0.90]{helvet} 
\usepackage{courier} 
\normalfont
\usepackage[T1]{fontenc}
\usepackage[textwidth=15cm,hcentering]{geometry}
\usepackage[colorlinks=true,linkcolor=red,citecolor=blue]{hyperref}
\usepackage{tikz}

\DeclareMathAlphabet{\mathcal}{OMS}{cmsy}{m}{n}

\newtheorem{theo}{Theorem}[section]
\newtheorem{lemm}[theo]{Lemma}
\newtheorem{prop}[theo]{Proposition}

\theoremstyle{definition}
\newtheorem{defi}[theo]{Definition}
\newtheorem{cons}[theo]{Construction}

\newtheorem{rem}[theo]{Remark}

\newtheorem*{theo*}{Theorem}

\numberwithin{equation}{section}

\newcommand{\cat}{\mathbf}
\newcommand{\on}{\operatorname}
\newcommand{\id}{\mathrm{id}}

\newcommand{\Map}{\on{Map}}
\newcommand{\Hom}{\mathsf{Hom}}

\newcommand{\Aut}{\mathsf{Aut}}

\newcommand{\C}{\mathbb{C}}
\newcommand{\Q}{\mathbb{Q}}
\newcommand{\Z}{\mathbb{Z}}
\newcommand{\R}{\mathbb{R}}

\newcommand{\GT}{\mathrm{GT}}

\newcommand{\od}{\mathscr{D}} 
\newcommand{\og}{\mathscr{G}rav} 
\newcommand{\opab}{\mathscr{P}a\mathscr{B}} 

\renewcommand{\mod}{\mathcal{M}} 

\title{On two chain models for the gravity operad}

\author{Cl\'{e}ment Dupont}
\author{Geoffroy Horel}

\begin{document}

\address{Institut Montpelli\'{e}rain Alexander Grothendieck, CNRS, Univ. Montpellier, France.}
\email{clement.dupont@umontpellier.fr}

\address{Université Paris 13, Sorbonne Paris Cité, Laboratoire Analyse, Géométrie et Applications, CNRS (UMR 7539), 93430, Villetaneuse, France.}
\email{horel@math.univ-paris13.fr}


\begin{abstract}
In this note we recall the construction of two chain level lifts of the gravity operad, one due to Getzler--Kapranov and one due to Westerland. We prove that these two operads are formal and that they indeed have isomorphic homology.
\end{abstract}

\maketitle

\tableofcontents

\section{Introduction}

The gravity operad is an operad which was introduced by Getzler in \cite{getzlertopologicalgravity,getzleroperadsmodulispaces}. It is an operad in graded vector spaces over the rationals whose arity $n$ space is given by $H_{*-1}(\mod_{0,n+1})$, where $\mod_{0,n+1}$ is the moduli space of genus zero Riemann surfaces with $n+1$ marked points\footnote{Here and throughout this article, $H_*(X)$ denotes the singular homology with rational coefficients of a topological space $X$.}. Getzler gives two seemingly different descriptions of the operadic structure. 

On the one hand, there exists an injective transfer map $H_{*-1}(\mod_{0,n+1})\to H_*(\od(n))$ where $\od(n)$ denotes the arity $n$ space of the operad of little $2$-disks. This comes from the fact the $\mod_{0,n+1}$ is homotopy equivalent to the quotient of $\od(n)$ by the action of the circle $S^1$. Getzler observes that the collection of subspaces $H_{*-1}(\mod_{0,n+1})$ is stable under operadic composition and thus inherits an operad structure from the operad of little $2$-disks.

On the other hand, one can consider the Deligne--Mumford compactification $\overline{\mod}_{0,n+1}$ of $\mod_{0,n+1}$. The complement $\overline{\mod}_{0,n+1}-\mod_{0,n+1}$ is a normal crossing divisor which induces a stratification of $\overline{\mod}_{0,n+1}$ indexed by trees. The strata of codimension $1$ are isomorphic to products $\mod_{0,r+1}\times\mod_{0,s+1}$ with $r+s=n+1$, and we thus get residue morphisms
\[H^*(\mod_{0,n+1})\to H^{*-1}(\mod_{0,r+1}\times\mod_{0,s+1})\]
which, after dualization, can be shown to give  an operad structure on the collection of graded vector spaces $H_{*-1}(\mod_{0,n+1})$.

Each of these two definitions of the gravity operad can be lifted to the chain level. For the first definition, this was done by Westerland in \cite{westerlandequivariant}. We can consider the operad in chain complexes $C_*(\od)$ and observe that it supports an action of the group $S^1$. Taking homotopy fixed points in a suitably rigid way, we can construct an operad $C_*(\od)^{hS^1}$ equipped with a map $C_*(\od)^{hS^1}\to C_*(\od)$. Moreover Westerland observes that the homology of $C_*(\od)^{hS^1}$ together with its map to the homology of $C_*(\od)$ exactly recovers the definition of the gravity operad as a suboperad of $H_*(\od)$. 

The second definition can also be lifted to the level of chains, as observed by Getzler and Kapranov \cite{getzlerkapranovmodular}. Indeed, using differential forms with logarithmic singularities, the residue map can be modeled at the chain level. This allows one to construct a cooperad in the category of nuclear Fréchet spaces given in arity $n$ by the cochain complex $E^{*-1}(\overline{\mod}_{0,n+1},\log\partial\overline{\mod}_{0,n+1})$. Note that we have to work with a completed tensor product for the K\"unneth formula to hold at the chain level.

The goal of the present paper is to explain how these models of the gravity operad and their homology are related. Firstly, we prove that the two chain models for the gravity operad are formal, i.e. that they are quasi-isomorphic to their homology as operads. Secondly, we prove that the two models have isomorphic homology. These results combined show that all four operads contain essentially the same homotopical information. The second result is nothing but the equivalence of the two definitions of the gravity operad, whose proof we could not find in the literature, although it has been implicitly used in many references. On the one hand, the description as a suboperad of $H_*(\od)$ leads \cite{getzlertopologicalgravity} to give a presentation of the gravity operad; on the other hand, the description in terms of residue maps shows \cite{getzleroperadsmodulispaces} that the gravity operad is Koszul dual to the hypercommutative operad -- the operad structure on the collection of graded vector spaces $H_*(\overline{\mod}_{0,n+1})$ coming from gluing curves along marked points. For this reason, we believe that this comparison between the two definitions, although unsurprising to experts, is a useful addition to the literature.

Let us say a few words about the proofs of the two formality results. For the Westerland model, we use a criterion due to Sullivan in the context of differential graded algebras. The idea is to lift a grading automorphism of the homology of our operad (i.e. an automorphism that acts in homological degree $n$ by multiplication by $\alpha^n$ for some unit $\alpha$ of infinite order) to an automorphism at the level of chains that can then be used to produce a splitting of the chain operad. In order to do this we need a large supply of automorphisms of our operad. In fact, we construct an action of the $\mathbb{Q}$-points of the Grothendieck--Teichmüller group $\GT$ on the Westerland model. Using the surjectivity of the cyclotomic character map $\GT(\mathbb{Q})\to\Q^\times$, we obtain the desired lifting. For the Getzler--Kapranov model in terms of differential forms with logarithmic singularities, we recall the folklore proof of formality, which uses Deligne's mixed Hodge theory. The purity of the mixed Hodge structure on the cohomology of the spaces $\mod_{0,n+1}$ implies that the subcomplex of holomorphic differential forms has zero differential and still computes the cohomology of $\mod_{0,n+1}$. Therefore, we get an explicit suboperad with zero differential and which is such that the inclusion is a quasi-isomorphism.

	\subsection*{Notations and conventions}
	
	A $n$-tree is a reduced rooted tree with leaves labeled by $\{1,\ldots,n\}$. For $X$ a topological space, we denote by $H_*(X)$ (resp. $H^*(X)$) the homology (resp. cohomology) groups of $X$ with coefficients in $\Q$. By convention, our operads do not have arity $0$ operations.
	
	\subsection*{Acknowledgements}
	
	We thank the Max-Planck-Institut f\"{u}r Mathematik (Bonn), where this project started, for providing excellent working conditions. We also thank Dan Petersen for several helpful comments about a first draft of this paper.

\section{The Westerland model}

	\subsection{The spectral model}

		Let $\od$ be the little $2$-disks operad. This is an operad in the category of topological spaces. The space $\od(n)$ has the $\Sigma_n$-equivariant homotopy type of the space of ordered configurations of $n$ points in the plane. The operad $\od$ possesses an action of the circle. There is a weak equivalence $\od(n)/S^1\simeq \mod_{0,n+1}$ for $n\geq 2$ where $\mod_{0,n+1}$ is the moduli space of genus $0$ curves with $n+1$ marked points. Note that, since the action of $S^1$ on the space $\od(n)$ is free for $n\geq 2$, this quotient really is a homotopy quotient.
		
		\begin{prop}\label{prop: Klein}
		Let $X$ be a spectrum with an $S^1$-action that is induced (i.e., weakly equivalent as an $S^1$-spectrum to $Y\wedge \Sigma^\infty_+S^1$ for some spectrum $Y$). Then the norm map
		\[\Sigma X_{hS^1}\to X^{h S^1}\]
		is a weak equivalence.
		\end{prop}
		
		\begin{proof}
		This is classical. See for instance \cite[Theorem D]{kleindualizing}.
		\end{proof}
		
		It is easy to verify that $\Sigma_+^\infty\od(n)$ is induced for $n\geq 2$. In fact this is already true at the space level, since the space $\od(n)$ is weakly equivalent to $S^1\times\mod_{0,n+1}$. It follows that there is an equivalence
		\[\Sigma\Sigma_+^\infty\mod_{0,n+1}\simeq\Sigma\Sigma_+^\infty\od(n)_{hS^1}
		\xrightarrow{\sim} (\Sigma_+^\infty \od(n))^{hS^1}\ .\]
		Since the functor $X\mapsto X^{hS^1}$ can be made lax monoidal, the spectra $(\Sigma_+^\infty \od(n))^{hS^1}$ form an operad in spectra. Let $H\Q$ denote the rational Eilenberg--MacLane spectrum. 

		\begin{defi}
		The collection of spectra $H\Q\wedge (\Sigma_+^\infty \od(n))^{hS^1}$ form an operad in $H\Q$-modules, that we call the \emph{Westerland spectral model of the gravity operad}.
		\end{defi}		
		
		By the above discussion, this operad has the homotopy type of $H\Q\wedge\Sigma\Sigma_+^\infty\mod_{0,n+1}$ in arity $\geq 2$ and is given by $H\Q$ in arity $1$. Note that since the spectrum $H\Q\wedge\Sigma_+^\infty\od(n)$ is also $S^1$-induced, there is a weak equivalence
		\[ (H\Q\wedge \Sigma\Sigma_+^\infty \od(n))_{hS^1}\simeq (H\Q\wedge \Sigma_+^\infty \od(n))^{hS^1}\ .\]
for $n\geq 2$. This immediately implies the following proposition.
		
		\begin{prop}\label{prop: fixed points commute HQ}
		There is a weak equivalence of operads
		\[H\Q\wedge((\Sigma^\infty_+\od)^{hS^1})\xrightarrow{\sim} (H\Q\wedge(\Sigma^\infty_+\od))^{hS^1}\ .\]
		\end{prop}
		
		\subsection{The chain complex model}
		
		The homotopy theory of $H\Q$-modules is naturally equivalent to that of chain complexes over $\Q$ as was established by Schwede and Shipley (see \cite[Theorem 5.1.6]{schwedestable}). This equivalence can be made symmetric monoidal as proved in \cite[Theorem 7.1.2.13]{luriehigher}. Therefore, the Westerland spectral model $H\Q\wedge((\Sigma^\infty_+\od)^{hS^1})$ of the gravity operad corresponds to an operad in chain complexes which is uniquely defined up to quasi-isomorphism. By Proposition \ref{prop: fixed points commute HQ} this operad in chain complexes should be defined as $C_*(\od)^{hS^1}$ where $C_*$ is our notation for the singular chain complex with rational coefficients. The only difficulty is to make sense of this homotopy fixed point construction in a rigid enough way, so that $C_*(\od)^{hS^1}$ is indeed an operad. A chain complex with an $S^1$-action can be defined as a chain complex with an action of the cdga $C_*(S^1)$. The singular chains of any topological space with an $S^1$-action will possess this structure. The problem is that the category of chain complexes with such an action does not form a symmetric monoidal category because the cdga $C_*(S^1)$ is not a Hopf algebra on the nose. One way to get around this difficulty is to use the theory of $\infty$-categories. In order to make this note more self-contained, we have chosen a different and more concrete route using simplicial $\Q$-vector spaces.
		
		We denote by $N$ the functor that assigns to a simplicial $\Q$-vector space its normalized chain complex. For $X$ a simplicial set, we denote by $S_\bullet(X)$ the simplicial vector space whose $n$-simplices is the free $\Q$-vector space with basis $X_n$. If $X$ is a topological space, we denote by $S_\bullet(X)$ the simplicial $\Q$-vector space $S_\bullet(\on{Sing}(X))$. The functor $S_\bullet$ is strong monoidal. It follows that $S_\bullet(S^1)$ is a simplicial Hopf algebra and moreover the functor $S_\bullet$ induces a symmetric monoidal functor from the category of spaces with an $S^1$-action to the category of simplicial modules over $S_\bullet(S^1)$.
		
		Given two simplicial vector spaces $X$ and $Y$, we denote by $\Hom(X,Y)$ the simplicial vector space whose degree $n$ simplices are the linear maps $X\otimes S_\bullet(\Delta[n])\to Y$ where $\Delta[n]$ is the simplicial set represented by $[n]$. 
		
		\begin{cons}\label{cons: construction of derived Hom}
		Let $A$ be a simplicial algebra. Let $M$ and $N$ be two simplicial left modules over $A$. We can form the cosimplicial simplicial module given by
		\[[n]\mapsto \Hom(A^{\otimes n}\otimes M,N)\]
		Let us explain how the two cofaces $\Hom(M,N)\to \Hom(A\otimes M,N)$  and the codegeneracy $\Hom(A\otimes M,N)\to \Hom(M,N)$ are defined, the higher cofaces and codegeneracies will be clear from that. The first coface is the map $\Hom(M,N)\to\Hom(A\otimes M,N)$ induced by the action $A\otimes M\to M$, the second coface is the following composition
		\[\Hom(M,N)\xrightarrow{A\otimes-}\Hom(A\otimes M,A\otimes N)\to \Hom(A\otimes M,N)\]
		where the second map is induced by the action of $A$ on $N$. Finally the codegeneracy $\Hom(A\otimes M,N)\to \Hom(M,N)$ is given by precomposition with the map
		\[\id_M\otimes u:M\to A\otimes M
		\]
		where $u:\Q\to A$ is the unit of $A$.
		
		We define $\R\Hom_A(M,N)$ to be the totalization of this cosimplicial simplicial vector space. This is a simplicial vector space. Note that, as suggested by the notation, the functor $\R\Hom_A(-,-)$ is indeed a right derived functor of $\Hom_A(-,-)$ in the sense that it preserves weak equivalences in both variables and coincides with $\Hom_A(-,-)$ when the source is a free $A$-module.
		
		Now, if $M$ and $N$ are two chain complexes with an action of a dga $A$, we can define a similar cosimplicial object in chain complexes
		\[[n]\mapsto \Hom(A^{\otimes n}\otimes M,N)\]
		Its totalization (i.e., the total complex of the associated double complex) is denoted $\R\Hom_A(M,N)$.
		\end{cons}
		
		\begin{cons}\label{cons: fixed points}
		Now, we assume that $H$ is a cocommutative simplicial Hopf algebra. The category of simplicial $H$-modules becomes a symmetric monoidal category under the levelwise tensor product of $\Q$-vector spaces. Moreover, the augmentation $H\to \Q$ makes $\Q$ into a module over $H$. It is then easy to verify that the construction $M\mapsto \R\Hom_H(\Q,M)$ is a lax symmetric monoidal functor of the variable $M$. When $G$ is a topological monoid, the simplicial vector space $S_\bullet(G)$ is  a cocommutative Hopf algebra. For $M$ a module over $S_\bullet(G)$, we use the notation $M^{hG}$ instead of $\R\Hom_{S_\bullet(G)}(\Q,M)$.
		\end{cons}
		
		Applying this construction to the operad $S_\bullet(\od)$, we obtain an operad $S_\bullet(\od)^{hS^1}$ in the category of simplicial vector spaces. 
		
		\begin{defi}
		The \emph{Westerland chain model $\og^W$ of the gravity operad} is the operad in chain complexes $N S_\bullet(\od)^{hS^1}$. Its homology is denoted by $\on{Grav}^W$.
		\end{defi}
		
		By construction, this operad comes equipped with a map
		\[\iota:\og^W\to C_*(\od):=NS_\bullet(\od)\]
		We now study the effect of this map on homology. As the homology of a chain complex with an action of $C_*(S^1)$, the homology $H_*(\od(n))$ has an action of the exterior algebra $H_*(S^1)\cong \Q[\Delta]/(\Delta^2)=:\Lambda[\Delta]$ where $\Delta$ has degree $1$. Equivalently, the homology of $H_*(\od(n))$ is equipped with a cohomological differential $\Delta$. Our construction of $\og^W$ involves taking the totalization of a cosimplicial simplicial vector space. Hence, we get a spectral sequence computing the homology of $\og^W$ of the form
		\[\on{E}_2^{s,t}=\on{Ext}^s_{\Lambda[\Delta]}(\Q,H_{t}(\od(n)))\implies H_{t-s}(\og^W(n))
		\]
But as explained in \cite[Lemma 6.2]{westerlandequivariant} the homology $H_*(\od(n))$ is free over $\Lambda[\Delta]$. It follows that all the higher Ext terms are zero and we deduce that $H_k(\og^W(n))$ is the kernel of the operator $\Delta$ acting on $H_k(\od(n))$, recovering the definition of the gravity operad from \cite{getzlertopologicalgravity}.
		
		\subsection{Spectral model vs. chain model}
		
		In this subsection we outline an argument that shows that the operad $\og^W$ is indeed a chain complex model for the operad $(H\Q\wedge\Sigma^\infty_+\od)^{hS^1}$ introduced in the first subsection. As explained there, one would like to construct the homotopy fixed points for the $S^1$-action on $C_*(\od)$ in the category of operads in chain complexes. What we have done instead is take the homotopy fixed points of the $S^1$-action on $S_\bullet(\od)$ in the category of operads in simplicial vector spaces. The category of simplicial vector spaces is equivalent to the category of non-negatively graded chain complexes by a theorem of Dold and Kan. Moreover, we have an adjunction
		\[i:\cat{Ch}_*(\Q)_{\geq 0}\leftrightarrows \cat{Ch}_*(\Q):t_{\geq 0}
		\]
		between non-negatively graded chain complexes and chain complexes in which the left adjoint is the inclusion and the right adjoint sends $C_*$ to $\cdots \rightarrow C_2 \rightarrow C_1 \rightarrow Z_0$. Both adjoints are lax monoidal, therefore this adjunction induces an adjunction
		\[i:\cat{OpCh}_*(\Q)_{\geq 0}\leftrightarrows \cat{OpCh}_*(\Q):t_{\geq 0}
		\]
		between the corresponding categories of operads. Since both $i$ and $t_{\geq 0}$ preserve quasi-isomorphisms, we deduce that $t_{\geq 0}$ preserves homotopy limits. It follows from this discussion that the operad $\og^W$ is modeling $t_{\geq 0}(C_*(\od)^{hS^1})$. But by Proposition \ref{prop: Klein} and Proposition \ref{prop: fixed points commute HQ}, we know that the spectra $(H\Q\wedge \Sigma^\infty_+\od(n))^{hS^1}$ are connective. Using the equivalence between the homotopy theory of $H\Q$-modules and chain complexes, this can be translated by saying that $C_*(\od(n))^{hS^1}$ has homology concentrated in non-negative degrees. It follows that the map
		\[t_{\geq 0}(C_*(\od)^{hS^1})\to C_*(\od)^{hS^1}\]
		is aritywise a quasi-isomorphism and hence is a quasi-isomorphism of operads.

\subsection{An alternative model}

		We denote by $\GT$ the Grothendieck--Teichm\"uller group. This is a proalgebraic group over $\mathbb{Q}$ that fits in a short exact sequence
		\[1\to \GT^1\to\GT\xrightarrow{\chi}\mathbb{G}_m\to 1\]
		The map $\chi:\GT\to\mathbb{G}_m$ is called the cyclotomic character. The group $\GT^1$ is a pro-unipotent group.
		
		In this subsection, we will construct a differential graded operad $\og^{W'}$ that is equipped with an action of $\GT(\Q)$ and that is quasi-isomorphic to Westerland's operad $\og^W$. This action will be used to prove the formality of $\og^{W'}$ and hence also of $\og^W$ in the next subsection. A similar method was used by Petersen in \cite{petersenformality} in order to prove the formality of the little 2-disks operad.
		
		We start with the operad $\opab$ of parenthesized braids. This is an operad in groupoids (its definition can be found in Section 3.1 of \cite{tamarkinformality}). Applying the classifying space functor $B$ aritywise, one gets an operad $B\opab$ in simplicial sets that is weakly equivalent to $\on{Sing}(\od)$ by \cite[Section 3.2]{tamarkinformality}. Let us denote by $\mathbb{Z}$ the abelian group $\mathbb{Z}$ seen as a groupoid with a unique object. This has the structure of a group object in groupoids. In particular it makes sense to say that a groupoid $C$ has an action of $\mathbb{Z}$. This means that there is a morphism of groupoids $\mathbb{Z}\times C\to C$ that satisfies the usual axioms. The operad in groupoids $\opab$ has an action of $\mathbb{Z}$ that is described explicitly in \cite[III 5.2]{fressehomotopy2}. Applying the classifying space functor, we get an action of $B\Z$ on $B\opab$. Up to homotopy, this action is nothing but the  action of $S^1$ on the space of configurations of points in the plane.
		
		Given a group $G$, its prounipotent completion is the universal prounipotent algebraic group $\Gamma$ over $\Q$ equipped with a map $G\to \Gamma(\Q)$. This can be constructed explicitly as the prounipotent group associated to the Lie algebra of primitive elements in the completed group algebra $\Q[G]^{\wedge}$. This construction has been extended to groupoids and operads in groupoids in \cite[Chapter 9]{fressehomotopy1}. The prounipotent completion of $\opab$ is denoted $\opab_\Q$. The action of $\Z$ on $\opab$ induces an action of $\Q$  on $\opab_\Q$ (here $\Q$ denotes the one-object groupoid whose group of arrows is $\Q$, it is also the prounipotent completion of the groupoid $\Z$). This implies that $B\opab_{\Q}$ has an action of $B\Q$ and that the operad $S_\bullet(\opab_\Q)$ is an operad in simplicial modules over the simplicial Hopf algebra $S_\bullet(B\Q)$.  We denote by $\og^{W'}$ the operad $N(S_\bullet(B\opab_\Q)^{hB\Q})$ (see Construction \ref{cons: fixed points}).
		
		The operad $\opab_\Q$ has an action of the group $\GT(\Q)$ (see \cite[Theorem 11.1.7]{fressehomotopy1}). Thus we have an action of $\GT(\Q)$ on $B\opab_\Q$ that is moreover compatible with the action of $B\Q$ in the sense that the action map
		\begin{equation}\label{eqn: action map is equivariant}
		B\Q\times B\opab_\Q(n)\to B\opab_\Q(n)
		\end{equation}
		is equivariant, where the left hand side is given the diagonal action and where we let $\GT(\Q)$ act on $\Q$ through the cyclotomic character (see \cite[Proposition III.5.2.4]{fressehomotopy2}). This implies that the cosimplicial object that enters in the definition of $\og^{W'}$ has a levelwise action of $\GT(\Q)$ that commutes with the cofaces and codegeneracies and hence that the operad $\og^{W'}$ has an action of $\GT(\Q)$ which is such that the map $\og^{W'}\to C_*(B\opab_\Q)$ is $\GT(\Q)$-equivariant.

		Now, we want to prove that $\og^W$ is quasi-isomorphic to $\og^{W'}$. This will rely on the following general lemma about model categories.

		\begin{lemm}\label{lemm: invariance of fixed points}
		Let $\cat{M}$ be a combinatorial simplicial model category. Let $C$ be a small simplicial category. Assume that for each object $c$ of $C$, the inclusion $i_c:\Map_C(c,c)\to C$ is a Dwyer-Kan equivalence of simplicial categories (where a monoid is seen as a category with one object). Let $F:C\to\cat{M}$ be a simplicial functor. Then, the objects $F(c)^{h\Map(c,c)}$ for $c\in C$ are all weakly equivalent.
		\end{lemm}
		
		\begin{proof}
		Let $c$ be an object of $C$. We have an inclusion $i_c:\Map_C(c,c)\to C$. By hypothesis, the map $i_c$ is an equivalence of simplicial categories, therefore, the adjunction 
		\[i_c^*:\cat{M}^C\leftrightarrows \cat{M}^{\Map_C(c,c)}:(i_c)_*\]
		is a Quillen equivalence of model categories (where both sides are given the injective model structure) by \cite[Proposition A.3.3.8]{luriehighertopos}. It follows that for any $F$ in $\cat{M}^{C}$ the derived unit map $F\to (\R i_c)_*i_c^*F$ is a weak equivalence. We can apply the functor $\on{holim}_C$ to this weak equivalence and we get a weak equivalence
		\[\on{holim}_CF\to \on{holim}_C (\R i_c)_*i_c^*F\]
		and the right-hand side can be identified with $\on{holim}_{\Map_C(c,c)}i_c^*F:=F(c)^{h\Map_C(c,c)}$. Therefore, all the objects $F(c)^{h\Map_C(c,c)}$ are weakly equivalent to $\on{holim}_CF$. 
		\end{proof}
		
		\begin{prop}
		The operad $\og^W$ is quasi-isomorphic to $\og^{W'}$.
		\end{prop} 
		
		\begin{proof}
		First, the map $B\opab\to B\opab_\Q$ induces a weak equivalence on rational homology. Moreover it is $B\Z$-equivariant (where $\Z$ acts on the target through the inclusion $\Z\to\Q$). Hence it induces a weak equivalence of operads
		\[S_\bullet(B\opab)^{hB\Z}\to S_\bullet(B\opab_\Q)^{hB\Z}\]
		
		The inclusion $\Z\to \Q$ induces a map 
		\[S_\bullet(\opab_\Q)^{hB\Z}\to S_\bullet(\opab_\Q)^{hB\Q}\]
		which is also a weak equivalence since the map of Hopf algebras $S_\bullet(B\Z)\to S_\bullet(B\Q)$ is a weak equivalence.

		Hence, it is enough to prove that $S_\bullet(B\opab)^{hB\Z}$ is equivalent to $S_\bullet(\od)^{hS^1}$ as an operad in simplicial vector spaces. In order to simplify the notations, we write $\mathscr{B}$ for the operad $B\opab$. We may assume without loss of generality that $\mathscr{B}$ and $\od $ are cofibrant-fibrant objects in simplicial operads. Thus, there exists a weak equivalence $\alpha:\mathscr{B}\to\od$ and a homotopy inverse $\beta:\od\to\mathscr{B}$. We denote by $C$ the simplicial subcategory of the category of simplicial operads containing the two objects $\mathscr{B}$ and $\od$ and the connected components of the map $\id_{\mathscr{B}}$, $\id_{\od}$, $\alpha$, $\beta$. The simplicial category $C$ has the property that for any object $c\in C$, the inclusion $\Map_C(c,c)\to C$ is a weak equivalence of simplcial categories. There is a simplicial functor from $C$ to operads in simplicial vector spaces sending $\mathscr{B}$ to $S_\bullet(\mathscr{B})$ and $\od$ to $S_\bullet(\od)$. Hence according to Lemma \ref{lemm: invariance of fixed points}, there is a zig-zag of weak equivalences:
		\[S_\bullet(\mathscr{B})^{h\Map_{C}(\mathscr{B},\mathscr{B})}\xleftarrow{\sim}*\xrightarrow{\sim}S_\bullet(\od)^{h\Map_{C}(\od,\od)}\]
		Finally since the inclusions $B\Z\to\Map_C(\mathscr{B},\mathscr{B})$ and $S^1\to\Map_C(\od,\od)$ are weak equivalences of monoids by \cite[Theorem 8.5]{horelprofinite}, the left-hand side of this zig-zag is weakly equivalent to $S_\bullet(\mathscr{B})^{hB\Z}$ and the right-hand side of this zig-zag is weakly equivalent to $S_\bullet(\od)^{hS^1}$.
		\end{proof}
		
		\subsection{Formality}

		Given an operad $\mathscr{P}$ (or any other algebraic structure) in graded vector spaces over $\Q$ and an element $r\in\Q^\times$, we get an automorphism $\alpha_r$ of $\mathscr{P}$ via the formula 
		\[\alpha_r(x):=r^{|x|}x.\]
		Such automorphisms are called grading automorphisms. Note that we have the formula $\alpha_r\circ\alpha_s=\alpha_{rs}$. Hence, the operad $\mathscr{P}$ has an action of the group $\Q^\times$ 
		
		\begin{defi} 
		This action of $\Q^\times$ on operads in graded vector spaces is called the grading action. More generally, an action of $\GT(\Q)$ on an operad $\mathscr{P}$ in graded vector spaces is said to be the grading action if it is given by the composition
		\[\GT(\Q)\xrightarrow{\chi} \Q^\times\to \Aut(\mathscr{P})\]
		where the second map is the grading action.
		\end{defi}
		
		\begin{prop}\label{prop: grading action}
		The action of $\GT(\Q)$ on $H_*(\og^W)$ is the grading action.
		\end{prop}
		
		\begin{proof}
		As we explained at the end of section 2.2, the map
		\[\iota(n):\og^W(n)\to C_*(\od(n))\]
		induces the inclusion $\ker(\Delta)\to H_*(\od(n))$ on homology groups. Since $\GT(\Q)$ acts on $\od(n)$ in a way compatible with the $S^1$-action, the map $H_*(\iota(n))$ is $\GT(\Q)$-equivariant. As explained in \cite{petersenformality}, the action of $\GT(\Q)$ on $H_*(\od(n))$ is the grading action; it follows that the action on $H_*(\og^W(n))$ is also the grading action.
		\end{proof}
		
		We can now prove the main result of this section.
		
		\begin{theo}
		The operad $\og^W$ is formal.
		\end{theo}
		
		\begin{proof}
		It is equivalent to prove that $\og^{W'}$ is formal. According to \cite[Theorem 5.2.3]{guillenmoduli}, it suffices to prove that a grading automorphism of $H_*(\og^{W'})$ lifts to an automorphisms of $\og^{W'}$. This follows immediately from the surjectivity of the cyclotomic character $\GT(\Q)\to\Q^\times$.
		\end{proof}
		
		\begin{rem}
		We conclude this section with a remark which connects this proof of formality to the one of the next section. The group $\GT$ receives a map from the group $\on{Gal}(\cat{MT}(\mathbb{Z}))$, the Galois group of the Tannakian category of mixed Tate motives over $\mathbb{Z}$ (see \cite[25.9.2.2]{andrebook}). By restricting along this map, the operad $\og^{W'}$ can be viewed as an operad in mixed Tate motives over $\mathbb{Z}$. As such it has a Hodge realization, which is an operad in the category of chain complexes in mixed Hodge structures. In this framework, the analog of Proposition \ref{prop: grading action} means that the induced mixed Hodge structure on homology is pure of weight $-2k$ in homological degree $k$ (see Remark \ref{rem: tate twist} below). Thus, our proof of formality can be reinterpreted in that light as a ``purity implies formality'' type of result. We refer the reader to \cite[Section 7.4]{ciricimixed} for more details about this.
		\end{rem}

\section{The Getzler--Kapranov model}

	\subsection{Definition}

		We recall the construction of \cite[\S 6.10]{getzlerkapranovmodular} in the genus zero case. Let $\mod_{0,n+1}$ denote the moduli space of genus zero curves with $n+1$ marked points and let $\overline{\mod}_{0,n+1}$ denote its Deligne--Mumford compactification. The complement $\partial\overline{\mod}_{0,n+1}:=\overline{\mod}_{0,n+1}-\mod_{0,n+1}$ is a simple normal crossing divisor which induces a stratification of $\overline{\mod}_{0,n+1}$ indexed by the poset of $n$-trees. One associates to integers $r$, $s$ such that $r+s=n+1$, and an integer $i\in\{1,\ldots,r\}$, a $n$-tree $t(r,s,i)$ with one internal edge obtained by grafting a $s$-corolla on the $i$-th leaf of a $r$-corolla. Figure \ref{fig: tree} shows the case $r=6$, $s=3$, $i=3$. This $n$-tree corresponds to an irreducible component of the divisor $\partial\overline{\mod}_{0,n+1}$, isomorphic to $\overline{\mod}_{0,r+1}\times\overline{\mod}_{0,s+1}$
		
		\begin{figure}[h]
		
		\begin{tikzpicture}
			\node {} [grow'=up, level distance=4mm, level/.style={sibling distance=12mm/#1}]
			child{[level distance=12mm]
				child{ node {$1$}}
				child{ node {$2$}}
				child { 
					child{ node {$3$}}
					child{ node {$4$}}
					child{ node {$5$}}
				}
				child{ node {$6$}}
				child{ node {$7$}}
				child{ node {$8$}}
			};
		\end{tikzpicture}
		
		\caption{The tree $t(6,3,3)$}\label{fig: tree}
		\end{figure}
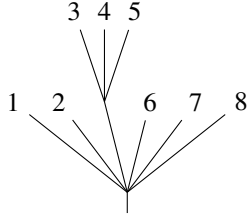

		We denote by 
		$$E^*(\overline{\mod}_{0,n+1},\log\partial\overline{\mod}_{0,n+1})$$
		the space of global smooth differential forms on $\overline{\mod}_{0,n+1}$ with logarithmic singularities along $\partial\overline{\mod}_{0,n+1}$. The residue morphism along the divisor $\overline{\mod}_{0,r+1}\times\overline{\mod}_{0,s+1}$ indexed by the tree $t(r,s,i)$ reads
		\begin{equation}\label{eqn: residue E}
		E^{*+1}(\overline{\mod}_{0,n+1},\log\partial\overline{\mod}_{0,n+1}) \rightarrow E^{*}(\overline{\mod}_{0,r+1}\times \overline{\mod}_{0,s+1},\log (\partial\overline{\mod}_{0,r+1}\times\overline{\mod}_{0,s+1} \cup \overline{\mod}_{0,r+1}\times\partial\overline{\mod}_{0,s+1}))\ .
		\end{equation}
		
		We now view the spaces of differential forms as nuclear Fr\'{e}chet spaces. Recall \cite[Proposition 3.0.6]{costello} that the category of nuclear Fréchet spaces, endowed with the completed tensor product $\widehat{\otimes}$, is symmetric monoidal. The right-hand side of (\ref{eqn: residue E}) is then naturally isomorphic to the tensor product
		$$E^*(\overline{\mod}_{0,r+1},\log\partial\overline{\mod}_{0,r+1}) \;\widehat{\otimes}\; E^*(\overline{\mod}_{0,s+1},\log\partial\overline{\mod}_{0,s+1})\ .$$
		For $V$ a nuclear Fréchet space, its strong dual $V'$ is a nuclear DF-space and this operation establishes an anti-equivalence of symmetric monoidal categories between the category of nuclear Fréchet spaces and that of nuclear DF-spaces \cite[Proposition 3.0.6]{costello}. By dualizing (\ref{eqn: residue E}) and suspending we thus get morphisms
		\begin{equation}\label{eqn: operadic composition E}
		E^{*-1}(\overline{\mod}_{0,r+1},\log\partial\overline{\mod}_{0,r+1})' \;\widehat{\otimes}\; E^{*-1}(\overline{\mod}_{0,s+1},\log\partial\overline{\mod}_{0,s+1})' \rightarrow E^{*-1}(\overline{\mod}_{0,n+1},\log\partial\overline{\mod}_{0,n+1})'\ .
		\end{equation}
		
		\begin{defi} 
		The \emph{Getzler--Kapranov chain model $\og^{GK}$ of the gravity operad}  is the differential graded operad in DF-spaces whose arity $n$ component is 
		$$\og^{GK}(n):=E^{*-1}(\overline{\mod}_{0,n+1},\log\partial\overline{\mod}_{0,n+1})'$$
		and whose composition morphisms $\circ_i$ are the morphisms (\ref{eqn: operadic composition E}).
		\end{defi}
		
		\begin{rem}
		These operads have the structure of anticyclic operads \cite[2.10]{getzlerkapranovcyclic}. This point of view has the advantage of making more explicit the signs that appear in the definition of the composition morphisms.
		\end{rem}
		
		Let us mention that the inclusion of $E^*(\overline{\mod}_{0,n+1},\log\partial\overline{\mod}_{0,n+1})$ inside the differential graded algebra of smooth differential forms on $\mod_{0,n+1}$ is a quasi-isomorphism. This implies that the homology of $\og^{GK}$ has arity $n$ component
		$$H_*(\og^{GK}(n))\cong H_{*-1}(\mod_{0,n+1})\otimes_\Q\C\ .$$
		It is a standard fact that the residue morphisms are defined on the cohomology with rational coefficients (this follows for instance from Lemma \ref{lemm: residues blow-ups}); thus, there is a natural rational structure on the homology of $\og^{GK}$, that we denote by $\on{Grav}^{GK}$. This is an operad in rational graded vector spaces whose arity $n$ component is $$\on{Grav}^{GK}(n)=H_{*-1}(\mod_{0,n+1})\ .$$
		It is nothing but (the operadic desuspension of) the operad defined by Getzler in \cite[\S 3.4]{getzleroperadsmodulispaces}.
		
		\begin{rem}\label{rem: tate twist}
		The Getzler--Kapranov gravity operad $\on{Grav}^{GK}$ has a natural structure of an operad in the category of mixed Hodge structures if one adds the right Tate twist and sets
		$$\on{Grav}^{GK}(n)=H_{*-1}(\mod_{0,n+1})\otimes\mathbb{Q}(1)\ .$$
		The Tate twist $\mathbb{Q}(1)$ has the effect of shifting the weight filtration by $-2$. By \cite[Lemma 3.12]{getzleroperadsmodulispaces}, the mixed Hodge structure on the $k$-th cohomology group of $\mod_{0,n+1}$ is pure Tate of weight $2k$, which implies that the mixed Hodge structure on the degree $k$ part of $\on{Grav}^{GK}$ is pure Tate of weight $-2(k-1)-2=-2k$. From a more concrete point of view, the Tate twist comes from the factor $2\pi i$ in the definition of a residue morphism.
		\end{rem}
	
	\subsection{Formality}
	
		We start with a general proposition. Let $X$ be a smooth complex variety and $D$ be a simple normal crossing divisor in $X$. Then we have the space $E^*(X,\log D)$ of global smooth differential forms on $X$ with logarithmic singularities along $D$, and the subspace $\Omega^*(X,\log D)$ of global holomorphic differential forms on $X$ with logarithmic singularities. The following proposition seems to be folklore, and is explained in, e.g., \cite[\S 1.6]{almpetersen}.
	
		\begin{prop}\label{prop: general formality}
		\begin{enumerate}
		\item If $X$ is projective then every global holomorphic logarithmic differential form is closed, i.e., the differential in $\Omega^*(X,\log D)$ is zero.
		\item If, furthermore, for every $k$ the mixed Hodge structure on $H^k(X-D)$ is pure of weight $2k$, then the inclusion
		$$(\Omega^*(X,\log D),d=0) \hookrightarrow (E^*(X,\log D),d)$$
		is a quasi-isomorphism of differential graded algebras.
		\end{enumerate}
		\end{prop}
		
		\begin{proof}
		Let us denote by $E^*_X(\log D)$ (resp $\Omega^*_X(\log D)$) the complex of sheaves on $X$ of smooth (resp. holomorphic) differential forms with logarithmic sigularities along $D$, whose space of global sections is $E^*(X,\log D)$ (resp. $\Omega^*(X,\log D)$). The inclusion 
		\begin{equation}\label{eqn: quasi-iso Omega E}
		\Omega^*_X(\log D)\hookrightarrow E^*_X(\log D)
		\end{equation}
		is a quasi-isomorphism of complexes of sheaves on $X$ \cite[3.2.3 b)]{delignehodge2}.
		\begin{enumerate}
		\item By \cite[Corollaire 3.2.13 (ii)]{delignehodge2}, the hypercohomology spectral sequence for the stupid truncation filtration on $\Omega^*_X(\log D)$ degenerates at $E_1$. The $E_1$ term is $E_1^{p,q}=H^q(X,\Omega^p_X(\log D))$ and the differential $d_1^{p,q}$ is induced by the exterior differential on differential forms. Thus, the degeneration of this spectral sequence implies in particular that $d_1^{p,0}=0$, which implies the claim.
		\item Again by the degeneration of the spectral sequence, we have 
		\[E_1^{p,q}=H^q(X,\Omega^p_X(\log D)) \cong \mathrm{gr}^p_FH^{p+q}(X-D)\otimes_\Q\C.\]
		By the purity assumption, this is zero for $q>0$. Thus, the sheaves $\Omega^p_X(\log D)$ are acyclic. This is also true for the (soft) sheaves $E^p_X(\log D)$; thus, taking global sections of (\ref{eqn: quasi-iso Omega E}) leads to the desired quasi-isomorphism.
		\end{enumerate}
		\end{proof}

		We note that under the assumptions of Proposition \ref{prop: general formality} (2), the complement $X-D$ is formal in the sense of rational homotopy theory, i.e., its differential graded algebra of smooth differential forms $(E^*(X-D),d)$ is formal. This is because the inclusion $(E^*(X,\log D),d)\hookrightarrow (E^*(X-D),d)$ is a quasi-isomorphism of differential graded algebras. This applies in particular to $X=\overline{\mod}_{0,n+1}$ and $D=\partial\overline{\mod}_{0,n+1}$ since the complement $\mod_{0,n+1}$ satisfies the purity assumption \cite[Lemma 3.12]{getzleroperadsmodulispaces}. In this case Proposition \ref{prop: general formality} also implies the following \emph{operadic} formality result, which appears in \cite[\S 6.10]{getzlerkapranovmodular} and \cite[\S 1.6]{almpetersen}.
		
		\begin{theo}
		The operad $\og^{GK}$ is formal.
		\end{theo}
		
		\begin{proof}
		By Proposition \ref{prop: general formality} the inclusion
		$$(\Omega^*(\overline{\mod}_{0,n+1},\log \partial\overline{\mod}_{0,n+1}),d=0) \hookrightarrow (E^*(\overline{\mod}_{0,n+1},\log \partial\overline{\mod}_{0,n+1}),d)$$
		is a quasi-isomorphism and induces an isomorphism $\Omega^*(\overline{\mod}_{0,n+1},\log \partial\overline{\mod}_{0,n+1}) \cong H^*(\mod_{0,n+1})\otimes_\Q\C$. This inclusion is compatible with the residue morphisms since the residue of a holomorphic logarithmic form is holomorphic. We thus get, after dualizing and suspending, a quasi-isomorphism of operads $\og^{GK} \stackrel{\sim}{\rightarrow} \on{Grav}^{GK}\otimes_\Q\C$.
		\end{proof}
		
		\begin{rem}
		As noted in \cite{getzlerkapranovmodular}, the same argument implies that $\og^{GK}$ is formal as an anticyclic operad.
		\end{rem}

\section{Comparing the two definitions of the gravity operad}

	The missing link between the two definitions of the gravity operad that we have used is a third definition given in \cite{kimurastasheffvoronov}.

	\subsection{Models with corners}
	
		For an integer $n\geq 2$ let us denote by $C(n)=\on{Conf}(n,\R^2)/(\R^2\rtimes \R_{>0})$ the quotient of the configuration space of $n$ ordered points in $\R^2$ by translations and dilations. There is a natural $S^1$-action on $C(n)$, whose quotient map is the natural map $C(n)\rightarrow \mod_{0,n+1}$. Here we briefly explain how to construct a commutative square 
		
		$$\xymatrix{
		C(n) \; \ar@{^{(}->}[r]^-{\sim} \ar[d] & FM(n) \ar[d] \\
		\mod_{0,n+1} \ar@{^{(}->}[r]^-{\sim} & X(n)
		}$$
		where $FM(n)$ and $X(n)$ are compactifications of $C(n)$ and $\mod_{0,n+1}$ respectively which are homotopy equivalences, the top horizontal arrow is $S^1$-equivariant, and the vertical arrows are the quotient maps.\\
		
		The space $FM(n)$ is the Fulton--MacPherson compactification of $C(n)$, which was introduced in the context of operads by Getzler--Jones \cite{getzlerjones}. Let us recall that it is a manifold with corners whose interior is $C(n)$, and that it has a natural stratification indexed by the poset of $n$-trees. The stratum corresponding to a $n$-tree $t$ is denoted by $FM^0(t)$, and its closure is denoted by $FM(t)$. They have codimension the number of internal edges of $t$, and we have natural product decompositions

		\begin{equation}\label{eqn: product decompositions}
		FM^0(t)\simeq \prod_{v\in V(t)} C(|v|) \;\;\; \textnormal{ and } \;\;\; FM(t)\simeq \prod_{v\in V(t)} FM(|v|)\ ,
		\end{equation}
where $V(t)$ denotes the set of vertices of $t$, and $|v|$ denotes the number of incoming edges at a vertex $v$. The $S^1$-action on $FM(n)$ is compatible with the stratifications, and the induced action on the products \eqref{eqn: product decompositions} is the diagonal action. This shows that the quotient $X(n):=FM(n)/S^1$ has the structure of a manifold with corners, and has a stratification indexed by the poset of $n$-trees. The interior of $X(n)$ is $\mod_{0,n+1}$, and the compactification $\mod_{0,n+1}\hookrightarrow X(n)$ can alternatively be obtained from $\overline{\mod}_{0,n+1}$ by performing real blow-ups of all irreducible components of the boundary $\partial\overline{\mod}_{0,n+1}$. For instance, $X(3)$ is isomorphic to the real blow-up of $\mathbb{P}^1(\C)$ along three points. For more details, see \cite{kimurastasheffvoronov}, where $X(n)$ is denoted by $\underline{\mathcal{M}}_{n+1}$, and \cite{kontsevichbourbaki} , where it is denoted by $\overline{\mathcal{M}}_{0,n+1}^{\R}$.\\
		
		It is customary to set $C(0)=FM(0)=\varnothing$ and $C(1)=FM(1)=\{*\}$. By using the product decompositions \eqref{eqn: product decompositions}, one sees that the closed immersions $FM(t)\hookrightarrow FM(n)$ give the collection $\{FM(n)\, , \, n\geq 0\}$ the structure of a topological operad. This is a model for the little disks operad, as the following proposition shows.
		
		\begin{prop}{\cite{getzlerjones,kontsevichoperadsmotives,salvatoresummable,lambrechtsvolic}}\label{prop: qiso D FM}
		The topological operads $FM$ and $\od$ are connected by a zig-zag of weak equivalences.
		\end{prop}
		
		In the next section we explain how to get the structure of an operad on the shifted homology groups of the spaces $X(n)$.

	\subsection{The Kimura--Stasheff--Voronov operad}
	
		Let us denote by $X^0(t)$ the stratum of $X(n)$ corresponding to a rooted $n$-tree $t$, and by $X(t)$ its closure. We have natural isomorphisms:
		$$X^0(t) \simeq \left(\prod_{v\in V(t)} C(|v|)\right) / S^1 \;\;\; \textnormal{ and } \;\;\; X(t) \simeq \left(\prod_{v\in V(t)} FM(|v|)\right) / S^1$$
		where the quotients refer to the diagonal $S^1$-actions. Thus, $X(t)$ is acted upon by the topological group $(S^1)^{V(t)}/S^1$, and the quotient map is
		$$X(t) \rightarrow \prod_{v\in V(t)} X(|v|)\ .$$
		For instance, for the tree $t=t(r,s,i)$ (see Figure \ref{fig: tree}), $X(t)\hookrightarrow X(n)$ is a closed subspace of real codimension $1$ and we get a $S^1$-bundle
		\begin{equation}\label{eqn: S1 bundle X}
		X(t)\rightarrow X(r)\times X(s)\ .
		\end{equation}
		In homology, this gives rise to a transfer map
		\begin{equation}\label{eqn: composition for KSR}
		H_*(X(r))\otimes H_*(X(s)) \stackrel{\simeq}{\longrightarrow} H_*(X(r)\times X(s)) \rightarrow H_{*+1}(X(t)) \rightarrow H_{*+1}(X(n))\ ,
		\end{equation}
		where the first map is the K\"{u}nneth isomorphism, the second map is the transfer map associated to the $S^1$-bundle \eqref{eqn: S1 bundle X}, and the third map is induced by the inclusion $X(t)\hookrightarrow X(n)$.
		
		\begin{defi}
		The \emph{Kimura--Stasheff--Voronov gravity operad} $\on{Grav}^{KSV}$ is the graded operad whose arity $n$ component is 
		$$\on{Grav}^{KSV}(n):=H_{*-1}(X(n))$$ 
		and whose composition morphisms $\circ_i$ are the (suspensions of the) morphisms \eqref{eqn: composition for KSR}.
		\end{defi}
		
	\subsection{Compatibility with residues}

		We start with a general lemma. Let $X$ be a smooth complex variety, $D\subset X$ be a smooth divisor, and $\pi:Y\rightarrow X$ be the real blow-up along $D$. It is a manifold with boundary $\partial Y=\pi^{-1}(D)$. The restriction $\pi:\partial Y\rightarrow D$ is a $S^1$-bundle which is nothing but the sphere bundle of the normal bundle of $D$ inside $X$. We thus have a transfer map $H_{*-1}(D)\rightarrow H_*(\partial Y)$ in homology. We also have a map $H_*(\partial Y)\rightarrow H_*(Y)$ induced by the closed immersion $\partial Y\hookrightarrow Y$ in homology, and we note that the inclusion $Y-\partial Y\rightarrow Y$ is a homotopy equivalence, and that the restriction $\pi:Y-\partial Y\rightarrow X-D$ is an isomorphism.
		
		\begin{lemm}\label{lemm: residues blow-ups}
		The composite
		$$H_{*-1}(D) \rightarrow H_*(\partial Y) \rightarrow H_*(Y) \simeq H_*(Y-\partial Y) \simeq H_*(X-D)$$
		is dual to the residue morphism $H^*(X-D)\rightarrow H^{*-1}(D)$.
		\end{lemm}
		
		\begin{proof}
		It is enough to do the proof in the case of homology and cohomology with complex coefficients, in which case it is a consequence of the Leray residue formula, see \cite[Theorem 2.4]{phambook}.
		\end{proof}
	
		\begin{prop}\label{prop: iso GK KSV}
		The natural isomorphism $H_{*-1}(\mod_{0,n+1}) \stackrel{\simeq}{\longrightarrow} H_{*-1}(X(n))$ induces an isomorphism of operads between the homology of the Getzler--Kapranov chain model for the gravity operad and the Kimura--Stasheff--Voronov gravity operad:
		$$\on{Grav}^{GK}\stackrel{\simeq}{\longrightarrow} \on{Grav}^{KSV}\ .$$
		\end{prop}
		
		\begin{proof}
		We show that the isomorphisms are compatible with the composition maps $\circ_i$ corresponding to the tree $t=t(r,s,i)$ (see Figure \ref{fig: tree}). It is convenient to set $\mod_{0,n+1}^+=\mod_{0,n+1}\cup\mod_{0,r+1}\times\mod_{0,s+1}$ and $X(n)^+=X^0(n)\cup X^0(t)$, viewed as open subspaces of $\overline{\mod}_{0,n+1}$ and $X(n)$, respectively. By construction, there is a morphism $X(n)^+\rightarrow \mod_{0,n+1}^+$, which is the real blow-up along $\mod_{0,r+1}\times\mod_{0,s+1}$. In the following diagram, the arrows marked $\tau$ are transfer maps for $S^1$-bundles and the arrows marked $i_*$ are induced in homology by closed immersions. According to Lemma \ref{lemm: residues blow-ups}, the first row of the diagram is the composition morphism $\circ_i$ in the operad $\on{Grav}^{GK}$. 
		$$\xymatrix{
		H_{*-2}(\mod_{0,r+1}\times\mod_{0,s+1}) \ar[r]^-{\tau} \ar[d]^-{\simeq} & H_{*-1}(X^0(t)) \ar[r]^-{i_*} \ar@{..>}[d]^-{\simeq} & H_{*-1}(X(n)^+) \ar@{..>}[d]^-{\simeq} & H_{*-1}(\mod_{0,n+1}) \ar[l]^-{\simeq} \ar[dl]^-{\simeq} \\
		H_{*-2}(X(r)\times X(s)) \ar[r]^-{\tau} & H_{*-1}(X(t)) \ar[r]^{i_*} & H_{*-1}(X(n)) & 
		}$$
		The arrows marked $\simeq$ in this diagram are induced by open immersions which are homotopy equivalences, and the diagram commutes. Since the second row is the composition morphism $\circ_i$ in the operad $\on{Grav}^{KSV}$, we are done.
		\end{proof}

	\subsection{Compatibility with the little disks}
	
		The quotient map $FM(n)\rightarrow X(n)$ is a $S^1$-bundle and thus gives rise to a transfer map in homology
		\begin{equation}\label{eqn: transfer X FM}
		H_{*-1}(X(n)) \rightarrow H_*(FM(n))\ .
		\end{equation}	
	
		\begin{prop}\label{prop: KSV FM}
		The transfer map \eqref{eqn: transfer X FM} induces a morphism of operads from the Kimura--Stasheff--Voronov operad to the homology of the Fulton--MacPherson operad:
		$$\on{Grav}^{KSV} \rightarrow H_*(FM)\ .$$
		\end{prop}
		
		\begin{proof}
		We show that the transfer maps \eqref{eqn: transfer X FM} are compatible with the composition maps $\circ_i$ corresponding to the tree $t=t(r,s,i)$ (see Figure \ref{fig: tree}). This amounts to showing that the outer square of the following diagram commutes. The arrows marked $K$ are K\"{u}nneth isomorphisms, the arrows marked $\tau$ are transfer maps for torus bundles, and the arrows marked $i_*$ are induced in homology by closed immersions.
		$$\xymatrix{
		H_*(FM(r))\otimes H_*(FM(s)) \ar[r]^-{K}_-{\simeq} & H_*(FM(r)\times FM(s)) \ar[r]^-{=} & H_*(FM(t)) \ar[r]^-{i_*} & H_*(FM(n)) \\
		H_{*-1}(X(r))\otimes H_{*-1}(X(s)) \ar[r]_-{K}^-{\simeq} \ar[u]^-{\tau\otimes\tau} & H_{*-2}(X(r)\times X(s)) \ar[r]_-\tau \ar@{..>}[u]^-{\tau }& H_{*-1}(X(t))\ar[r]_-{i_*} \ar@{..>}[u]^-{\tau} & H_{*-1}(X(n)) \ar[u]^-{\tau}
		}$$
		It is enough to show that the three squares forming the diagram commute.
		\begin{enumerate}
		\item The leftmost square commutes because transfer maps are compatible with the K\"{u}nneth isomorphisms.
		\item The central square commutes because of the functoriality of the transfer maps for the composite 
		$FM(r)\times FM(s) \rightarrow (FM(r)\times FM(s))/S^1 = X(t) \rightarrow (FM(r)/S^1) \times (FM(s)/S^1) = X(r)\times X(s)$.
		\item The rightmost square commutes because the following square is cartesian.
		$$\xymatrix{
		FM(t) \ar@{^{(}->}[r] \ar[d] & FM(n) \ar[d] \\
		X(t) \ar@{^{(}->}[r] & X(n)	
		}$$
		\end{enumerate}
		\end{proof}
		
	\subsection{Equivalence of the two definitions of the gravity operad}
	
		\begin{theo}
		The natural isomorphisms $\on{Grav}^{GK}(n) \stackrel{\simeq}{\longrightarrow} \on{Grav}^W(n)$ induce an isomorphism of operads between the homology of the Getzler--Kapranov model and the homology of the Westerland model.
		\end{theo}
		
		\begin{proof}
		We form the following commutative square of symmetric sequences.
		$$\xymatrix{
		\on{Grav}^{GK} \ar[d]^-{\simeq} \ar[r]^-{(1)}_-{\simeq} & \on{Grav}^{KSV} \ar@{_{(}->}[d]^-{(2)} \\
		\on{Grav}^W \ar@{^{(}->}[r]_-{(3)} & H_*(FM)
		}$$
		The arrow labeled (1) is an isomorphism of operads by Proposition \ref{prop: iso GK KSV}; the arrow labeled (2) is a morphism of operads by Proposition \ref{prop: KSV FM}; the arrow labeled (3) is a morphism of operads by the construction of $\on{Grav}^W$ and Proposition \ref{prop: qiso D FM}. Thus, the remaining arrow is an isomorphism of operads.
		\end{proof}

\bibliographystyle{alpha}

\bibliography{biblio}

\end{document}